\documentclass[11pt,a4paper]{article}

\usepackage{inputenc}
\usepackage{amsmath}
\usepackage{bm}
\usepackage{bbold}
\usepackage{amsthm}
\usepackage{enumerate}

\usepackage[pdfstartview={XYZ null null 1.00},pdfpagemode=UseNone,breaklinks]{hyperref}

\usepackage{tikz}\tikzset{x=1cm,y=1cm,z=1cm}
\usetikzlibrary{perspective}

\usepackage{pgfplots}\pgfplotsset{compat=1.16}

\title{Algebraic solution to box-constrained\\ bi-criteria problem of rating alternatives through\\ pairwise comparisons\thanks{Kybernetika 58(5), 665-690 (2022). https://www.kybernetika.cz/content/2022/5/665}}

\author{N. Krivulin\thanks{Faculty of Mathematics and Mechanics, Saint Petersburg State University, 28 Universitetsky Ave., St.~Petersburg, 198504, Russia, nkk@math.spbu.ru.}}

\date{}

\newtheorem{theorem}{Theorem}
\newtheorem{lemma}[theorem]{lemma}

\theoremstyle{definition}
\newtheorem{example}{Example}

\begin{document}

\maketitle

\begin{abstract}
We consider a decision-making problem to evaluate absolute ratings of alternatives that are compared in pairs according to two criteria, subject to box constraints on the ratings. The problem is formulated as the log-Chebyshev approximation of two pairwise comparison matrices by a common consistent matrix (a symmetrically reciprocal matrix of unit rank), to minimize the approximation errors for both matrices simultaneously. We rearrange the approximation problem as a constrained bi-objective optimization problem of finding a vector that determines the approximating consistent matrix, and then represent the problem in terms of tropical algebra. We apply methods and results of tropical optimization to derive an analytical solution of the constrained problem. The solution consists in introducing two new variables that describe the values of the objective functions and allow reducing the problem to the solution of a system of parameterized inequalities constructed for the unknown vector, where the new variables play the role of parameters. We exploit the existence condition for solutions of the system to derive those values of the parameters that belong to the Pareto front inherent to the problem. Then, we solve the system for the unknown vector and take all solutions that correspond to the Pareto front, as a complete solution of the bi-objective problem. We apply the result obtained to the bi-criteria decision problem under consideration and present illustrative examples.
\\

\\

\end{abstract}

\section{Introduction}

The problem of evaluating alternatives based on their pairwise comparisons under multiple criteria (the multicriteria pairwise comparison problem) is one of the most common and demanded decision-making tasks in practice \cite{Gavalec2015Decision,Ramesh2013Multiple,Ramik2020Pairwise,Saaty1990Analytic}. 
Although there are a number of methods proposed to solve the problem, they may not be quite effective due to some drawbacks, which are widely discussed in the literature \cite{Barzilai1997Deriving,Belton1983Shortcoming}. As a result, the development of new solution techniques that can complement and supplement existing methods continues to be of theoretical interest and practical value. Of particular importance are solutions to new problems under additional assumptions, which are difficult to solve by existing methods, including constrained pairwise comparison problems.

\subsection{One criterion pairwise comparison problems}

First, consider an unconstrained pairwise comparison problem with a set of alternatives and one criterion. The alternatives are compared in pairs, which results in a matrix $\bm{A}=(a_{ij})$ where $a_{ij}$ shows the relative preference (priority, weight) of alternative $i$ over $j$. Given a matrix $\bm{A}$ of pairwise comparisons, the problem is to evaluate a vector $\bm{x}=(x_{i})$ of absolute ratings (scores, weights) of the alternatives under examination.

The entries of a pairwise comparison matrix $\bm{A}$ must satisfy the condition $a_{ij}=1/a_{ji}$ for all $i$ and $j$, and hence the matrix $\bm{A}$ is symmetrically reciprocal. According to this condition (which is assumed to hold in practice), if alternative $i$ is $a_{ij}$ times preferred to $j$, than alternative $j$ has to be $1/a_{ij}$ times preferred to $i$.

A symmetrically reciprocal matrix $\bm{A}$ is consistent if its entries satisfy the transitivity condition $a_{ij}=a_{ik}a_{kj}$ for all $i$, $j$ and $k$. This condition (which is commonly violated in real-world problems) suggests that if alternative $i$ is $a_{ik}$ times preferred to $k$ whereas alternative $k$ is $a_{kj}$ times preferred to $j$, then alternative $i$ should be $a_{ik}a_{kj}$ times preferred to $j$. If the pairwise comparison matrix $\bm{A}$ is consistent, then there exists a positive vector $\bm{x}=(x_{i})$ defined up to a positive factor, that determines the matrix $\bm{A}$ by the condition $a_{ij}=x_{i}/x_{j}$ and thus directly solves the pairwise comparison problem. As a solution vector $\bm{x}$ in this case, one can take any column of the matrix $\bm{A}$.

Since the pairwise comparison matrices encountered in practice are usually not consistent, the solution is to find a consistent matrix that is as much as possible close to (approximates) the given inconsistent matrix. Then, the positive vector corresponding to this consistent matrix is taken as an approximate solution of the problem.

The methods available to solve the one criterion pairwise comparison problem include both heuristic procedures and approximation techniques \cite{Choo2004Common,Saaty1984Comparison}. The heuristic procedures mainly offer various schemes of aggregating columns in the pairwise comparison matrix $\bm{A}$, such as the weighted column sum methods that derive the solution vector by calculating a weighted sum of columns in the matrix $\bm{A}$. In the case when the column weights are set proportional to the solution vector, the technique is known as the principal eigenvector method \cite{Saaty1990Analytic,Saaty2013Onthemeasurement}. This method takes the principal (Perron) eigenvector of $\bm{A}$ as a vector $\bm{x}$ of absolute ratings, and finds wide application in practice.

The approximation techniques are based on minimizing an error function that measures the distance between matrices when a pairwise comparison matrix $\bm{A}$ is approximated by a consistent matrix \cite{Saaty1984Comparison}. The most common approach involves approximation with the Euclidean distance in logarithmic scale (the log-Euclidean approximation), which leads to the geometric mean method \cite{Barzilai1997Deriving,Crawford1985Note}. This method yields an analytical solution where the elements of the vector $\bm{x}$ of ratings are calculated as the geometric means of the corresponding rows in the matrix $\bm{A}$. Both methods of principal eigenvector and geometric mean provide a unique (up to a positive factor) vector $\bm{x}$. They, however, do not allow to account for in a simple way additional constraints, such as box constraints on $\bm{x}$, which can arise as an essential condition in real-world problems.

Another approximation technique \cite{Elsner2004Maxalgebra,Elsner2010Maxalgebra} to solve the pairwise comparison problem, employs minimization of Chebyshev distance between matrices, taken in logarithmic scale (the log-Chebyshev approximation). In contrast to the principal eigenvector and geometric mean methods, the solution provided by the log-Chebyshev approximation can be nonunique. This may complicate the choice of one optimal solution in practice, but at the same time, can expand the possibilities of making decisions by taking into account additional conditions. Complete analytical solutions that describe in compact algebraic form all solutions for both unconstrained and constrained problems of pairwise comparisons are derived based on methods and results of tropical algebra in \cite{Krivulin2015Rating,Krivulin2016Using}.

The tropical (idempotent) mathematics, which is concerned with the theory and application of algebraic systems with idempotent operations \cite{Golan2003Semirings,Gondran2008Graphs,Heidergott2006Maxplus,Kolokoltsov1997Idempotent,Maclagan2015Introduction}, finds increasing use in various areas including operations research and management science. A binary operation is idempotent if performing the operation on two operands of the same value gives this value as the result (one can consider the operations $\max$ and $\min$ as the most relevant examples). Application of methods and results of tropical mathematics often allows one to find new effective solutions to old and novel problems in location analysis, project scheduling, decision making and other areas (see e.\,g., \cite{Elsner2004Maxalgebra,Elsner2010Maxalgebra,Goto2022Polyad,Gursoy2013Analytic}). In many cases, tropical algebra can offer direct complete analytical solutions for problems that has only numerical algorithmic solutions available in the literature.

\subsection{Multicriteria pairwise comparison problems}

The problem of evaluating alternatives from pairwise comparisons under multiple criteria corresponds to a more realistic scenario in the decision making process, while becomes more difficult to solve. The most commonly used solution is based on the heuristic technique known as the analytical hierarchy process method \cite{Saaty1977Scaling,Saaty1990Analytic,Saaty2013Onthemeasurement}. The method assumes that for given matrices $\bm{A}_{1},\ldots,\bm{A}_{k}$ of pairwise comparisons of alternatives according to $k$ criteria, and a matrix $\bm{C}$ of pairwise comparisons of criteria, one needs to find a common vector $\bm{x}$, which shows the absolute ratings of alternatives. The method calculates normalized (with respect to the sum of components) principal eigenvectors of the pairwise comparison matrices for alternatives and criteria. The solution vector $\bm{x}$ is then obtained as the sum of the eigenvectors of the matrices $\bm{A}_{1},\ldots,\bm{A}_{k}$, taken with weights equal to corresponding components of the eigenvector of the matrix $\bm{C}$.

Another widely adopted solution technique is the weighted geometric mean method \cite{Barzilai1997Deriving,Crawford1985Note}. The method involves calculating the vectors of geometric means for each matrix of pairwise comparisons involved in the problem. Furthermore, the components of the solution vector $\bm{x}$ are obtained as weighted products of corresponding components in the geometric mean vectors for matrices $\bm{A}_{1},\ldots,\bm{A}_{k}$, where the weights are given by the components of the normalized geometric mean vector for the matrix $\bm{C}$.

The most general approach to solve the multicriteria pairwise comparison problem under consideration is to handle this problem within the framework of multiobjective optimization \cite{Benson2009Multiobjective,Ehrgott2005Multicriteria,Luc2008Pareto,Pappalardo2008Multiobjective}. The solution of multiobjective optimization problems, where competing and conflicting objectives have to be satisfied, is usually given in the form of a set of Pareto-optimal (non-dominated) solutions, which are defined as solutions that cannot be improved in one objective without getting worse in another.

The derivation of all Pareto-optimal solutions in problems with many objectives is generally a rather difficult task. In this case, the solution is focused on finding one of these solutions, which is often obtained by reduction to an ordinary optimization problem with one objective by using scalarization techniques. The weighted geometric mean method presents an example of this approach, where the multiobjective problem is reduced to minimizing a weighted sum of the errors of simultaneous log-Euclidean approximation of the pairwise comparison matrices for criteria. As another example, one can consider the solution in terms of tropical algebra in \cite{Krivulin2016Using,Krivulin2019Tropical}, which is based on the minimization of a weighted maximum of errors in log-Chebyshev approximation.

If the number of objectives is small, say two or three, then a complete solution can sometimes be obtained through an analytical description of the Pareto front defined as the image of the set of Pareto-optimal solutions in the objective space. The description of the front is then used to derive all Pareto-optimal solutions in a parametric form. Specifically, this approach is applied in \cite{Krivulin2020Using} to derive all these solutions in a bi-criteria problem with two pairwise comparison matrices, which is formulated as simultaneous log-Chebyshev approximation of both matrices and then solved as a bi-objective optimization problem in the framework of tropical optimization. This result is extended in \cite{Krivulin2021Algebraic} to a bi-criteria problem where the solution vector of ratings of alternatives is subjected to relative constraints in the form of bounds on the ratios between the ratings.

Constrained multicriteria problems of rating alternatives from pairwise comparisons arise in different applications, where prior information about absolute or relative bounds on the ratings must be taken into account. Enabling constraints in the pairwise comparison problems reflects the increasing complexity of contemporary decision-making processes and may prove to be a very difficult task to attack with existing methods, including the analytical hierarchy process and weighted geometric mean methods.

In this paper, we present an analytical solution to a constrained bi-criteria pairwise comparison problem that accommodates box constraints on the vector $\bm{x}$ of ratings of alternatives. The problem is formulated as simultaneous log-Chebyshev approximation of two pairwise comparison matrices $\bm{A}$ and $\bm{B}$ by a common consistent matrix, to minimize approximation errors for the matrices as a vector objective function. We apply and further develop the tropical algebraic approach, which is implemented in \cite{Krivulin2020Using,Krivulin2021Algebraic} to solve both bi-objective pairwise comparison problems without constraints and with tropical linear inequality constraints.

We rearrange the approximation problem as a constrained bi-objective optimization problem of finding a vector that determines the approximating consistent matrix, and then represent the problem in terms of tropical algebra. The solution consists in introducing two new variables that describe the values of the objective functions and allow reducing the problem to the solution of a system of parameterized inequalities constructed for the unknown vector, where the new variables play the role of parameters. We exploit the existence condition for solutions of the system to derive those values of the parameters that belong to the Pareto front inherent to the problem. Then, we solve the system for the unknown vector and take all solutions that correspond to the Pareto front, as a complete solution of the bi-objective optimization problem.

The paper is arranged as follows. In Section~\ref{S-LCABS}, we outline the use of log-Chebyshev approximation to evaluate alternatives from pairwise comparisons. Section~\ref{S-PADR} includes preliminary notation, definitions and results of tropical algebra needed to represent and solve the optimization problem below. Section~\ref{S-BOOPWBC} presents a complete solution of a bi-objective tropical optimization problem with box constraints, given in terms of a linearly ordered, radicable tropical semifield. In Section~\ref{S-ABCBCDP}, we apply the obtained result to solve a bi-criteria pairwise comparison problem and give an illustrative example. Section~\ref{S-C} offers some concluding remarks on the results obtained.

\section{Log-Chebyshev approximation based solutions}
\label{S-LCABS}

Suppose that, given a positive symmetrically reciprocal matrix $\bm{A}=(a_{ij})$ of pairwise comparisons of $n$ alternatives, the problem is to find the vector $\bm{x}=(x_{j})$ of absolute ratings by means of approximation of the matrix $\bm{A}$ by a consistent matrix $\bm{X}=(x_{i}/x_{j})$. Approximation in the log-Chebyshev sense with a logarithm to a base greater than $1$ naturally leads us to solving the following optimization problem:
\begin{equation*}
\begin{aligned}
\min_{x_{1},\ldots,x_{n}}
&&&
\max_{1\leq i,j\leq n}
\left|
\log a_{ij}-\log\frac{x_{i}}{x_{j}}
\right|.
\end{aligned}
\end{equation*}

Let us verify that this problem reduces to minimizing a function without logarithms (see also \cite{Krivulin2020Using,Krivulin2019Tropical}). First, we apply the monotonicity of the logarithmic function to write
\begin{equation*}
\left|
\log a_{ij}-\log\frac{x_{i}}{x_{j}}
\right|
=
\max\left\{\
\log\frac{a_{ij}x_{j}}{x_{i}},\log\frac{x_{i}}{a_{ij}x_{j}}
\right\}
=
\log
\max\left\{\
\frac{a_{ij}x_{j}}{x_{i}},\frac{x_{i}}{a_{ij}x_{j}}
\right\}.
\end{equation*}

Then, with the condition $a_{ij}=1/a_{ji}$, the objective function in the problem becomes
\begin{equation*}
\max_{1\leq i,j\leq n}
\left|
\log a_{ij}-\log\frac{x_{i}}{x_{j}}
\right|
=
\log\max_{1\leq i,j\leq n}\max\left\{
\frac{a_{ij}x_{j}}{x_{i}},\frac{x_{i}}{a_{ij}x_{j}}
\right\}
=
\log\max_{1\leq i,j\leq n}\frac{a_{ij}x_{j}}{x_{i}}.
\end{equation*}

Finally, due to the monotonicity, the minimization of the logarithm is equivalent to minimizing its argument, which turns the approximation problem into the problem
\begin{equation*}
\begin{aligned}
\min_{x_{1},\ldots,x_{n}}
&&
\max_{1\leq i,j\leq n}\frac{a_{ij}x_{j}}{x_{i}}.
\end{aligned}
\end{equation*}

Note that any solution of the last problem also minimizes the maximum absolute relative error \cite{Elsner2004Maxalgebra,Krivulin2020Using,Portugal2011Weberfechner}, which is given by
\begin{equation*}
\max_{1\leq i,j\leq n}\frac{|a_{ij}-x_{i}/x_{j}|}{a_{ij}}.
\end{equation*}

Furthermore, we assume that $n$ alternatives are compared in pairs according to two criteria, which yields two matrices $\bm{A}=(a_{ij})$ and $\bm{B}=(b_{ij})$ of pairwise comparisons. The problem of evaluating the absolute ratings of alternatives takes the form of the approximation of these matrices by a common consistent matrix $\bm{X}=(x_{i}/x_{j})$. With the application of the log-Chebyshev approximation, we arrive at a bi-objective optimization problem to minimize the vector function
\begin{equation*}
\left(
\max_{1\leq i,j\leq n}\frac{a_{ij}x_{j}}{x_{i}},
\max_{1\leq i,j\leq n}\frac{b_{ij}x_{j}}{x_{i}}
\right).
\end{equation*}

Now, assume that there are box constraints imposed on the ratings of alternatives, given in the form of lower and upper bounds and written as the double inequalities
\begin{equation*}
g_{i}
\leq
x_{i}
\leq
h_{i},
\qquad
i=1,\ldots,n.
\end{equation*}

By combining the objective function with the constraints, we obtain the following constrained bi-objective optimization problem: given positive matrices $\bm{A}=(a_{ij})$ and $\bm{B}=(b_{ij})$, and vectors $\bm{g}=(g_{j})$ and $\bm{h}=(h_{j})$, find positive vectors $\bm{x}=(x_{j})$ that achieve
\begin{equation}
\begin{aligned}
\min_{x_{1},\ldots,x_{n}}
&&&
\left(
\max_{1\leq i,j\leq n}x_{i}^{-1}a_{ij}x_{j},
\max_{1\leq i,j\leq n}x_{i}^{-1}b_{ij}x_{j}
\right);
\\
\text{s.t.}
&&&
g_{j}
\leq
x_{j}
\leq
h_{j},
\qquad
j=1,\ldots,n.
\end{aligned}
\label{P-min_maxijaijxixj-maxijbijxixj-gileqxilehi}
\end{equation}

In the next sections, we show how problem \eqref{P-min_maxijaijxixj-maxijbijxixj-gileqxilehi} can be represented in terms of tropical algebra and then completely solved in the Pareto-optimal sense by using methods and results of tropical optimization.

\section{Preliminary algebraic definitions and results}
\label{S-PADR}

We start with an overview of basic notation, definitions and results of tropical (idempotent) algebra, which underlie the solution of the bi-objective optimization problem in the subsequent sections. For further detail on the theory and applications of tropical mathematics, one can consult, e.\,g. \cite{Golan2003Semirings,Gondran2008Graphs,Heidergott2006Maxplus,Kolokoltsov1997Idempotent,Maclagan2015Introduction}.

\subsection{Idempotent semifield}

Consider a nonempty set $\mathbb{X}$ equipped with the operations $\oplus$ (addition) and $\otimes$ (multiplication) together with their respective neutral elements $\mathbb{0}$ (zero) and $\mathbb{1}$ (one). Both operations $\oplus$ and $\otimes$ are associative and commutative, and multiplication distributes over addition. Addition has the idempotent property $x\oplus x=x$ for all $x\in\mathbb{X}$, and there are multiplicative inverses $x^{-1}$ such that $xx^{-1}=\mathbb{1}$ for all nonzero $x\in\mathbb{X}$ (here and henceforth, the multiplication sign $\otimes$ is omitted for the brevity sake). The algebraic system $(\mathbb{X},\oplus,\otimes,\mathbb{0},\mathbb{1})$ is commonly referred to as the idempotent (tropical) semifield.

The operation of exponentiation to integer powers is given by $x^{p}=xx^{p-1}$, $x^{-p}=(x^{-1})^{p}$, $\mathbb{0}^{p}=\mathbb{0}$ and $x^{0}=\mathbb{1}$ for any nonzero $x\in\mathbb{X}$ and integer $p>0$. Moreover, the semifield is considered radicable, which means that each $x$ has $p$th root for any integer $p>0$, and thus rational exponents are defined as well.

Idempotent addition provides a partial order that $x\leq y$ if and only if $x\oplus y=y$, which is extended to a total order on $\mathbb{X}$. The operations of addition and multiplication are monotonic in both arguments, which means that the inequality $x\leq y$ yields the inequalities $x\oplus z\leq y\oplus z$ and $xz\leq yz$. Exponentiation is monotonic in the sense that for nonzero $x,y$ and rational $q$, the inequality $x\leq y$ leads to $x^{q}\geq y^{q}$ if $q\leq0$ and to $x^{q}\leq y^{q}$ if $q>0$. Furthermore, addition satisfies an extremal property (the majority law) that the inequalities $x\leq x\oplus y$ and $y\leq x\oplus y$ hold for any $x,y$. Finally, the inequality $x\oplus y\leq x$ is equivalent to the system of two inequalities $x\leq z$ and $y\leq z$.

An example of the algebraic system under consideration is the idempotent semifield $\mathbb{R}_{\max}=(\mathbb{R}_{+},\max,\times,0,1)$ (also called the max-algebra), where $\mathbb{R}_{+}$ is the set of nonnegative reals. In the semifield, addition $\oplus$ is defined as taking maximum, multiplication is as usual, $\mathbb{0}=0$ and $\mathbb{1}=1$. The operation of exponentiation has the usual meaning, the order provided by idempotent addition corresponds to the natural linear order on $\mathbb{R}_{+}$.

Another example is the semifield $\mathbb{R}_{\min,+}=(\mathbb{R}\cup\{+\infty\},\min,+,+\infty,0)$ (the min-plus algebra), where $\mathbb{R}$ is the set of reals, which has $\oplus=\min$, $\otimes=+$, $\mathbb{0}=+\infty$ and $\mathbb{1}=0$. For each $x\ne\mathbb{0}$, there is the inverse $x^{-1}$, which is equal to the opposite number $-x$ in the usual arithmetic. The power $x^{y}$ coincides with the arithmetic product $xy$ and hence well-defined for any $x,y\in\mathbb{R}$. The order induced by the addition $\oplus$ is opposite to the standard linear order on $\mathbb{R}$.

\subsection{Matrices and vectors}

The algebra of matrices and vectors over $\mathbb{X}$ is introduced in the usual way with the set of matrices with $m$ rows and $n$ columns denoted by $\mathbb{X}^{m\times n}$, and the set of column vectors with $n$ components by $\mathbb{X}^{n}$. A matrix (vector) with all entries equal to $\mathbb{0}$ is the zero matrix (vector), which is denoted by $\bm{0}$. A matrix is called column-regular if it has no zero columns. A vector without zero components is called regular.

The matrix and vector operations follow the standard entrywise formulas, where the ordinary scalar addition and multiplication are replaced by the operations $\oplus$ and $\otimes$. The monotonicity properties of the scalar operations $\oplus$ and $\otimes$ are routinely extended to the matrix and vector operations, where the order relation is understood componentwise.

In the set of square matrices $\mathbb{X}^{n\times n}$, a matrix that has all diagonal entries equal to $\mathbb{1}$ and the other entries to $\mathbb{0}$ is the identity matrix denoted $\bm{I}$. Nonnegative integer powers of a square matrix $\bm{A}\ne\bm{0}$ indicate repeated multiplication given by $\bm{A}^{0}=\bm{I}$ and $\bm{A}^{p}=\bm{A}\bm{A}^{p-1}$ for any integer $p>0$. The trace of a matrix $\bm{A}=(a_{ij})$ is defined as
\begin{equation*}
\mathop\mathrm{tr}\bm{A}
=
a_{11}\oplus\cdots\oplus a_{nn}
=
\bigoplus_{i=1}^{n}a_{ii}.
\end{equation*}

For any conforming matrices $\bm{A}$ and $\bm{B}$, and scalar $x$, the trace satisfies the equalities
\begin{equation*}
\mathop\mathrm{tr}(\bm{A}\oplus\bm{B})
=
\mathop\mathrm{tr}\bm{A}
\oplus
\mathop\mathrm{tr}\bm{B},
\qquad
\mathop\mathrm{tr}(\bm{A}\bm{B})
=
\mathop\mathrm{tr}(\bm{B}\bm{A}),
\qquad
\mathop\mathrm{tr}(x\bm{A})
=
x\mathop\mathrm{tr}\bm{A}.
\end{equation*}

For any matrix $\bm{A}\in\mathbb{X}^{n\times n}$, a trace function is defined that takes the matrix to
\begin{equation*}
\mathop\mathrm{Tr}(\bm{A})
=
\mathop\mathrm{tr}\bm{A}
\oplus
\cdots
\oplus
\mathop\mathrm{tr}\bm{A}^{n}
=
\bigoplus_{k=1}^{n}\mathop\mathrm{tr}\bm{A}^{k}.
\end{equation*}

Provided that $\mathop\mathrm{Tr}(\bm{A})\leq\mathbb{1}$, the Kleene star operator maps the matrix $\bm{A}$ to the matrix
\begin{equation*}
\bm{A}^{\ast}
=
\bm{I}
\oplus
\bm{A}
\oplus
\cdots
\oplus
\bm{A}^{n-1}
=
\bigoplus_{k=0}^{n-1}\bm{A}^{k}.
\end{equation*}

For any nonzero column vector $\bm{x}=(x_{j})$, the multiplicative inverse transpose (conjugate) is a row vector $\bm{x}^{-}=(x_{j}^{-})$ where $x_{j}^{-}=x_{j}^{-1}$ if $x_{j}\ne\mathbb{0}$ and $x_{j}^{-}=\mathbb{0}$ otherwise.

The spectral radius of any matrix $\bm{A}\in\mathbb{X}^{n\times n}$ is given by
\begin{equation*}
\lambda
=
\mathop\mathrm{tr}\bm{A}
\oplus
\cdots
\oplus
\mathop\mathrm{tr}\nolimits^{1/n}(\bm{A}^{m})
=
\bigoplus_{k=1}^{n}\mathop\mathrm{tr}\nolimits^{1/k}(\bm{A}^{k}).
\end{equation*}

\subsection{Vector inequalities}

We turn to solutions of vector inequalities that provide a basis to handle the tropical optimization problems below. First assume that, given a matrix $\bm{A}\in\mathbb{X}^{m\times n}$ and a vector $\bm{d}\in\mathbb{X}^{m}$, the problem is to find all vectors $\bm{x}\in\mathbb{X}^{n}$ that satisfy the inequality
\begin{equation}
\bm{A}\bm{x}
\leq
\bm{d}.
\label{I-Axleqd}
\end{equation}

The solution of \eqref{I-Axleqd} is known under different assumptions in various forms. We use a solution described by the following statement (see, e.\,g. \cite{Krivulin2015Extremal}).
\begin{lemma}
\label{L-Axleqd}
For any column-regular matrix $\bm{A}$ and regular vector $\bm{d}$, all solutions of \eqref{I-Axleqd} are given by the inequality $\bm{x}\leq(\bm{d}^{-}\bm{A})^{-}$.
\end{lemma}

Suppose that for a given square matrix $\bm{A}\in\mathbb{X}^{n\times n}$ and vector $\bm{c}\in\mathbb{X}^{n}$, we find regular vectors $\bm{x}\in\mathbb{X}^{n}$ to solve the inequality
\begin{equation}
\bm{A}\bm{x}
\oplus
\bm{c}
\leq
\bm{x}.
\label{I-Axcleqx}
\end{equation}

A complete solution of the inequality is provided by the next result \cite{Krivulin2015Multidimensional}.
\begin{theorem}
\label{T-Axcleqx}
For any square matrix $\bm{A}$, the following statements are true.
\begin{enumerate}
\item
If $\mathop\mathrm{Tr}(\bm{A})\leq\mathbb{1}$, then all regular solutions of \eqref{I-Axcleqx} are given in parametric form by $\bm{x}=\bm{A}^{\ast}\bm{u}$ where $\bm{u}$ is a vector of parameters such that $\bm{u}\geq\bm{c}$.
\item
If $\mathop\mathrm{Tr}(\bm{A})>\mathbb{1}$, then there is no regular solution.
\end{enumerate}
\end{theorem}

Note that the parametric representation offered by the theorem describes all regular solutions as a subset of the linear span of the columns in the Kleene star matrix $\bm{A}^{\ast}$.

Consider the problem to solve for regular vectors $\bm{x}$ both inequalities \eqref{I-Axleqd} and \eqref{I-Axcleqx} simultaneously, which couple together to form one double inequality
\begin{equation}
\bm{A}\bm{x}
\oplus
\bm{c}
\leq
\bm{x}
\leq
\bm{d}.
\label{I-Axcleqxleqd}
\end{equation}

By combining the results of Lemma~\ref{L-Axleqd} and Theorem~\ref{T-Axcleqx}, we immediately obtain a complete solution of \eqref{I-Axcleqxleqd} (which can also be derived from a more general result in \cite{Krivulin2014Constrained}).
\begin{lemma}
\label{L-Axcleqxleqd}
For any matrix $\bm{A}$ such that $\mathop\mathrm{Tr}(\bm{A})\leq\mathbb{1}$ and regular vector $\bm{d}$, the following statements are true.
\begin{enumerate}
\item
If $\bm{d}^{-}\bm{A}^{\ast}\bm{c}\leq\mathbb{1}$, then all regular solutions of inequality \eqref{I-Axcleqxleqd} are given in parametric form by $\bm{x}=\bm{A}^{\ast}\bm{u}$ where $\bm{u}$ is a vector of parameters such that $\bm{c}\leq\bm{u}\leq(\bm{d}^{-}\bm{A}^{\ast})^{-}$.
\item
If $\bm{d}^{-}\bm{A}^{\ast}\bm{c}>\mathbb{1}$, then there is no regular solution.
\end{enumerate}
\end{lemma}
\begin{proof}
Under the condition $\mathop\mathrm{Tr}(\bm{A})\leq\mathbb{1}$, application of Theorem~\ref{T-Axcleqx} to the left inequality at \eqref{I-Axcleqxleqd} yields the solution $\bm{x}=\bm{A}^{\ast}\bm{u}$ where $\bm{u}\geq\bm{c}$. Substitution into the right inequality leads to an inequality for $\bm{u}$ in the form $\bm{A}^{\ast}\bm{u}\leq\bm{d}$. Observing that $\bm{A}^{\ast}\geq\bm{I}$, and thus $\bm{A}^{\ast}$ is column-regular, we apply Lemma~\ref{L-Axleqd} to obtain $\bm{u}\leq(\bm{d}^{-}\bm{A}^{\ast})^{-}$.

By combining both lower and upper bounds obtained for the vector $\bm{u}$, we arrive at the double inequality $\bm{c}\leq\bm{u}\leq(\bm{d}^{-}\bm{A}^{\ast})^{-}$. The last inequality determines a nonempty set of vectors $\bm{u}$ if and only if the condition $\bm{c}\leq(\bm{d}^{-}\bm{A}^{\ast})^{-}$ holds, which is equivalent by Lemma~\ref{L-Axleqd} to the inequality $\bm{d}^{-}\bm{A}^{\ast}\bm{c}\leq\mathbb{1}$.
\qed
\end{proof}

\subsection{Binomial identities}

We now present binomial identities for matrices (see also \cite{Krivulin2017Direct,Krivulin2020Using}), which are used to develop and represent further results. First, we note that for any square matrices $\bm{A},\bm{B}\in\mathbb{X}^{n\times n}$, straightforward algebra yields the following expansion formula:
\begin{equation}
(\bm{A}\oplus\bm{B})^{m}
=
\bm{A}^{m}
\oplus
\bigoplus_{k=1}^{m-1}
\bigoplus_{\substack{i_{0}+i_{1}+\cdots+i_{k}=m-k\\i_{0},i_{1},\ldots,i_{k}\geq0}}
\bm{A}^{i_{0}}(\bm{B}\bm{A}^{i_{1}}\cdots\bm{B}\bm{A}^{i_{k}})
\oplus
\bm{B}^{m}.
\label{E-AoplusBm}
\end{equation}

Consider the second term of the form of the double sum on the right-hand side. After summation of this term over all $m=2,\ldots,n$, we rearrange the summands in increasing order of $k$ to obtain the equality
\begin{multline*}
\bigoplus_{m=2}^{n}
\bigoplus_{k=1}^{m-1}
\bigoplus_{\substack{i_{0}+i_{1}+\cdots+i_{k}=m-k\\i_{0},i_{1},\ldots,i_{k}\geq0}}
\bm{A}^{i_{0}}(\bm{B}\bm{A}^{i_{1}}\cdots\bm{B}\bm{A}^{i_{k}})
\\=
\bigoplus_{k=1}^{n-1}
\bigoplus_{m=1}^{n-k}
\bigoplus_{\substack{i_{0}+i_{1}+\cdots+i_{k}=m\\i_{0},i_{1},\ldots,i_{k}\geq0}}
\bm{A}^{i_{0}}(\bm{B}\bm{A}^{i_{1}}\cdots\bm{B}\bm{A}^{i_{k}}).
\end{multline*}

Summing both sides of the expansion formula at \eqref{E-AoplusBm} for all $m=1,\ldots,n$, using the last equality and taking trace result in the binomial identity
\begin{equation}
\mathop\mathrm{Tr}(\bm{A}\oplus\bm{B})
=
\mathop\mathrm{Tr}(\bm{A})
\oplus
\mathop\mathrm{tr}
\bigoplus_{k=1}^{n-1}
\bigoplus_{m=1}^{n-k}
\bigoplus_{\substack{i_{0}+i_{1}+\cdots+i_{k}=m\\i_{0},i_{1},\ldots,i_{k}\geq0}}
\bm{A}^{i_{0}}(\bm{B}\bm{A}^{i_{1}}\cdots\bm{B}\bm{A}^{i_{k}})
\oplus
\mathop\mathrm{Tr}(\bm{B}).
\label{E-TrAoplusB}
\end{equation}

As a direct consequence of the identity, the following two inequalities are valid:
\begin{equation*}
\mathop\mathrm{Tr}(\bm{A}\oplus\bm{B})
\geq
\mathop\mathrm{Tr}(\bm{A}),
\qquad
\mathop\mathrm{Tr}(\bm{A}\oplus\bm{B})
\geq
\mathop\mathrm{Tr}(\bm{B}).
\end{equation*}

Suppose the condition $\mathop\mathrm{Tr}(\bm{A}\oplus\bm{B})\leq\mathbb{1}$ holds, and note that the inequalities $\mathop\mathrm{Tr}(\bm{A})\leq\mathbb{1}$ and $\mathop\mathrm{Tr}(\bm{B})\leq\mathbb{1}$ are then satisfied as well. In a similar way as before, we sum the second term on the right-hand side of \eqref{E-AoplusBm} over all $m=2,\ldots,n-1$, and rearrange the terms. As a result, we obtain the same equality as above with $n$ replaced by $n-1$. Finally, application of the Kleene star operator to \eqref{E-AoplusBm} leads to another binomial identity
\begin{equation}
(\bm{A}\oplus\bm{B})^{\ast}
=
\bm{A}^{\ast}
\oplus
\bigoplus_{k=1}^{n-2}
\bigoplus_{m=1}^{n-k-1}
\bigoplus_{\substack{i_{0}+i_{1}+\cdots+i_{k}=m\\i_{0},i_{1},\ldots,i_{k}\geq0}}
\bm{A}^{i_{0}}(\bm{B}\bm{A}^{i_{1}}\cdots\bm{B}\bm{A}^{i_{k}})
\oplus
\bm{B}^{\ast}.
\label{E-AoplusBast}
\end{equation}

\subsection{Polynomial functions}

We conclude with remarks about polynomial functions involved in the derivation of solutions presented below (see also \cite{Krivulin2021Algebraic}). Suppose that $a_{km}\in\mathbb{X}$ for all $k=1,\ldots,n-1$ and $m=1,\ldots, n-k$, and consider tropical polynomial functions on $\mathbb{X}$ with rational exponents (tropical Puiseux polynomials), which take any $s,t>\mathbb{0}$ to
\begin{equation*}
G(s)
=
\bigoplus_{k=1}^{n-1}
\bigoplus_{m=1}^{n-k}
a_{km}^{1/k}
s^{-m/k},
\qquad
H(t)
=
\bigoplus_{k=1}^{n-1}
\bigoplus_{m=1}^{n-k}
a_{km}^{1/m}
t^{-k/m}.
\end{equation*}

First, we note that both functions $G$ and $H$ monotonically decrease as their arguments increase. We now verify that these functions are inverse to each other. Indeed, if the equality $G(s)=t$ holds, then the following inequalities are valid as well:
\begin{equation*}
a_{km}^{1/k}
s^{-m/k}
\leq
t,
\qquad
k=1,\ldots,n-1;
\qquad
m=1,\ldots,n-k;
\end{equation*}
where at least one inequality holds as an equality.

Furthermore, we solve these inequalities for $s$ to obtain the inequalities
\begin{equation*}
a_{km}^{1/m}
t^{-k/m}
\leq
s,
\qquad
k=1,\ldots,n-1;
\qquad
m=1,\ldots,n-k.
\end{equation*}

After combining these inequalities, one of which is an equality, we get $H(t)=s$.

As a consequence, the inequalities $G(s)\leq t$ and $H(t)\leq s$ are equivalent since the solution of one of them is given by the other and vice versa.

\section{Bi-objective optimization problem with box constraints}
\label{S-BOOPWBC}

In this section, we examine a bi-objective tropical optimization problem where two multiplicative conjugate quadratic (pseudo-quadratic) forms are to be minimized subject to box constraints. We present a result that complements and extends solutions to the problems without constraints and with two-sided linear constraints, obtained in \cite{Krivulin2020Using,Krivulin2021Algebraic}.

Given matrices $\bm{A},\bm{B}\in\mathbb{X}^{n\times n}$ and vectors $\bm{g},\bm{h}\in\mathbb{X}^{n}$, the problem is to find regular vectors $\bm{x}\in\mathbb{X}^{n}$ that achieve
\begin{equation}
\begin{aligned}
\min_{\bm{x}}
&&&
(\bm{x}^{-}\bm{A}\bm{x},\ \bm{x}^{-}\bm{B}\bm{x});
\\
\text{s.t.}
&&&
\bm{g}
\leq
\bm{x}
\leq
\bm{h}.
\label{P-minxAxxBx-gleqxleqh}
\end{aligned}
\end{equation}

To solve the bi-objective problem, we introduce two auxiliary variables that represent the values of the objective functions, and then determine the Pareto front for the objectives. Furthermore, we reduce the problem to a bi-objective problem of minimizing the auxiliary variables subject to a parametrized system of inequality constraints on the vector $\bm{x}$ with these variables serving as parameters. The solution of the problem obtained involves two stages: first, the existence conditions for solutions of the system of constraints are used to evaluate the optimal values of the variables in the Pareto front, and second, all solutions $\bm{x}$ that correspond to the parameters given by the Pareto front are taken as a complete Pareto-optimal solution of the original problem.

\subsection{Solution of bi-objective problem}

We now formulate and prove the next result, which offers a complete direct Pareto-optimal solution to the constrained bi-objective optimization problem at \eqref{P-minxAxxBx-gleqxleqh}.

\begin{theorem}
\label{T-minxAxxBx-gleqxleqh}
Let $\bm{A}$ and $\bm{B}$ be nonzero matrices, $\bm{g}$ be a vector and $\bm{h}$ a regular vector such that $\bm{g}\leq\bm{h}$. Denote the spectral radii of $\bm{A}$ and $\bm{B}$ by
\begin{equation*}
\lambda
=
\bigoplus_{k=1}^{n}
\mathop\mathrm{tr}\nolimits^{1/k}(\bm{A}^{k}),
\qquad
\mu
=
\bigoplus_{k=1}^{n}
\mathop\mathrm{tr}\nolimits^{1/k}(\bm{B}^{k}),
\end{equation*}
and introduce the scalars
\begin{equation*}
\gamma
=
\bigoplus_{k=1}^{n-1}
(\bm{h}^{-}\bm{A}^{k}\bm{g})^{1/k},
\qquad
\delta
=
\bigoplus_{k=1}^{n-1}
(\bm{h}^{-}\bm{B}^{k}\bm{g})^{1/k}.
\end{equation*}
For all $s,t>\mathbb{0}$, define the functions
\begin{equation*}
\begin{aligned}
G(s)
&=
\bigoplus_{k=1}^{n-1}
\bigoplus_{m=1}^{n-k}
s^{-m/k}
\mathop\mathrm{tr}\nolimits^{1/k}(\bm{F}_{km})
\oplus
\bigoplus_{k=1}^{n-2}
\bigoplus_{m=1}^{n-k-1}
s^{-m/k}
(\bm{h}^{-}
\bm{F}_{km}
\bm{g})^{1/k},
\\
H(t)
&=
\bigoplus_{k=1}^{n-1}
\bigoplus_{m=1}^{n-k}
t^{-k/m}
\mathop\mathrm{tr}\nolimits^{1/m}(\bm{F}_{km})
\oplus
\bigoplus_{k=1}^{n-2}
\bigoplus_{m=1}^{n-k-1}
t^{-k/m}
(\bm{h}^{-}
\bm{F}_{km}
\bm{g})^{1/m},
\end{aligned}
\end{equation*}
where, for $m=1,\ldots,n-k$ and $k=1,\ldots,n-1$, the following notation is used:
\begin{equation*}
\bm{F}_{km}
=
\bigoplus_{\substack{i_{0}+i_{1}+\cdots+i_{k}=m\\i_{0},i_{1},\ldots,i_{k}\geq0}}
\bm{A}^{i_{0}}(\bm{B}\bm{A}^{i_{1}}\cdots\bm{B}\bm{A}^{i_{k}}).
\end{equation*}
Then, the following statements hold.
\begin{enumerate}
\item
If $H(\mu\oplus\delta)\leq\lambda\oplus\gamma$, then the Pareto front of problem \eqref{P-minxAxxBx-gleqxleqh} reduces to a single point $(\alpha,\beta)$ with the coordinates
\begin{equation*}
\alpha
=
\lambda\oplus\gamma,
\qquad
\beta
=
\mu\oplus\delta.
\end{equation*}
\item
If $H(\mu\oplus\delta)>\lambda\oplus\gamma$, then the Pareto front forms a segment of points $(\alpha,\beta)$ defined by the conditions
\begin{equation*}
\lambda\oplus\gamma
\leq
\alpha
\leq
H(\mu\oplus\delta),
\qquad
\beta
=
G(\alpha).
\end{equation*}
\item
All Pareto-optimal solutions of problem \eqref{P-minxAxxBx-gleqxleqh} are given in parametric form by
\begin{equation*}
\bm{x}
=
(\alpha^{-1}\bm{A}
\oplus
\beta^{-1}\bm{B})^{\ast}
\bm{u},
\end{equation*}
where $\bm{u}$ is a vector of parameters that satisfies the condition
\begin{equation*}
\bm{g}
\leq
\bm{u}
\leq
(\bm{h}^{-}
(\alpha^{-1}\bm{A}
\oplus
\beta^{-1}\bm{B})^{\ast})^{-}.
\end{equation*}
\end{enumerate}
\end{theorem}
\begin{proof}
We begin with the introduction of auxiliary scalar variables $\alpha$ and $\beta$ to rewrite problem \eqref{P-minxAxxBx-gleqxleqh} as the constrained bi-objective problem
\begin{equation}
\begin{aligned}
\min_{\bm{x},\alpha,\beta}
&&&
(\alpha,\beta);
\\
\text{s.t.}
&&&
\bm{x}^{-}\bm{A}\bm{x}
\leq
\alpha,
\quad
\bm{x}^{-}\bm{B}\bm{x}
\leq
\beta,
\quad
\bm{g}
\leq
\bm{x}
\leq
\bm{h}.
\end{aligned}
\label{P-minalphabeta-xAxleqalpha-xBxleqbeta-gleqxleqh}
\end{equation}

We solve the problem for the variables $\alpha$ and $\beta$ in the sense of Pareto optimality to determine the Pareto front in the objective space of $\alpha$ and $\beta$. Let us consider the system of inequality constraints in problem \eqref{P-minalphabeta-xAxleqalpha-xBxleqbeta-gleqxleqh}
\begin{equation*}
\bm{x}^{-}\bm{A}\bm{x}
\leq
\alpha,
\qquad
\bm{x}^{-}\bm{B}\bm{x}
\leq
\beta,
\qquad
\bm{g}
\leq
\bm{x}
\leq
\bm{h}.
\end{equation*}

Since the matrices $\bm{A}$ and $\bm{B}$ are nonzero, and the vector $\bm{x}$ is regular, we see that $\alpha\geq\bm{x}^{-}\bm{A}\bm{x}>\mathbb{0}$ and $\beta\geq\bm{x}^{-}\bm{B}\bm{x}>\mathbb{0}$.

We apply Lemma~\ref{L-Axleqd} to solve the first inequality in the system with respect to $\bm{A}\bm{x}$ and the second with respect to $\bm{B}\bm{x}$, and then rearrange the system as
\begin{equation*}
\alpha^{-1}
\bm{A}\bm{x}
\leq
\bm{x},
\qquad
\beta^{-1}
\bm{B}\bm{x}
\leq
\bm{x},
\qquad
\bm{g}
\leq
\bm{x}
\leq
\bm{h}.
\end{equation*}

After combining all inequalities in the system obtained, problem \eqref{P-minalphabeta-xAxleqalpha-xBxleqbeta-gleqxleqh} turns into
\begin{equation}
\begin{aligned}
\min_{\bm{x},\alpha,\beta}
&&&
(\alpha,\beta);
\\
\text{s.t.}
&&&
(\alpha^{-1}\bm{A}\oplus\beta^{-1}\bm{B})\bm{x}
\oplus
\bm{g}
\leq
\bm{x}
\leq
\bm{h}.
\end{aligned}
\label{P-minalphabeta-alpha1Aoplusbeta1Boplusgleqxleqh}
\end{equation}

We now examine the double inequality constraint in \eqref{P-minalphabeta-alpha1Aoplusbeta1Boplusgleqxleqh} by applying Lemma~\ref{L-Axcleqxleqd}. We use the assumption and condition of the lemma to derive the Pareto front as a set of points $(\alpha,\beta)$, and then take the corresponding set of solution vectors $\bm{x}$ offered by the lemma as the Pareto-optimal solution of problem \eqref{P-minxAxxBx-gleqxleqh}.

First, observe that under the assumption and condition of the lemma, all solutions of the double inequality are given by
\begin{equation*}
\bm{x}
=
(\alpha^{-1}\bm{A}
\oplus
\beta^{-1}\bm{B})^{\ast}
\bm{u},
\qquad
\bm{g}
\leq
\bm{u}
\leq
(\bm{h}^{-}
(\alpha^{-1}\bm{A}
\oplus
\beta^{-1}\bm{B})^{\ast})^{-}.
\end{equation*}

We begin with the assumption of the lemma, which takes the form of the inequality
\begin{equation*}
\mathop\mathrm{Tr}(\alpha^{-1}\bm{A}\oplus\beta^{-1}\bm{B})
\leq
\mathbb{1}.
\end{equation*}

We rearrange the left-hand side by using \eqref{E-TrAoplusB} with $\bm{A}$ replaced by $\alpha^{-1}\bm{A}$ and $\bm{B}$ by $\beta^{-1}\bm{B}$, and then apply the linearity of trace to write
\begin{equation*}
\bigoplus_{k=1}^{n}
\alpha^{-k}
\mathop\mathrm{tr}\bm{A}^{k}
\oplus
\bigoplus_{k=1}^{n-1}
\bigoplus_{m=1}^{n-k}
\alpha^{-m}\beta^{-k}
\mathop\mathrm{tr}\bm{F}_{km}
\oplus
\bigoplus_{k=1}^{n}
\beta^{-k}
\mathop\mathrm{tr}\bm{B}^{k}
\leq
\mathbb{1},
\end{equation*}
where we use the notation
\begin{equation*}
\bm{F}_{km}
=
\bigoplus_{\substack{i_{0}+i_{1}+\cdots+i_{k}=m\\i_{0},i_{1},\ldots,i_{k}\geq0}}
\bm{A}^{i_{0}}(\bm{B}\bm{A}^{i_{1}}\cdots\bm{B}\bm{A}^{i_{k}}).
\end{equation*}

It follows from properties of idempotent addition that the inequality can be separated into three inequalities, which yields the system
\begin{gather*}
\bigoplus_{k=1}^{n}
\alpha^{-k}
\mathop\mathrm{tr}\bm{A}^{k}
\leq
\mathbb{1},
\qquad
\bigoplus_{k=1}^{n}
\beta^{-k}
\mathop\mathrm{tr}\bm{B}^{k}
\leq
\mathbb{1},
\\
\bigoplus_{k=1}^{n-1}
\bigoplus_{m=1}^{n-k}
\alpha^{-m}\beta^{-k}
\mathop\mathrm{tr}\bm{F}_{km}
\leq
\mathbb{1}.
\end{gather*}

The system obtained is equivalent to the system of inequalities
\begin{gather*}
\alpha^{-k}
\mathop\mathrm{tr}\bm{A}^{k}
\leq
\mathbb{1},
\qquad
\beta^{-k}
\mathop\mathrm{tr}\bm{B}^{k}
\leq
\mathbb{1},
\qquad
k=1,\ldots,n;
\\
\bigoplus_{m=1}^{n-k}
\alpha^{-m}\beta^{-k}
\mathop\mathrm{tr}\bm{F}_{km}
\leq
\mathbb{1},
\qquad
k=1,\ldots,n-1.
\end{gather*}

By solving the first inequality for $\alpha$ and the other two for $\beta$, we have the inequalities
\begin{gather*}
\alpha
\geq
\mathop\mathrm{tr}\nolimits^{1/k}(\bm{A}^{k}),
\qquad
\beta
\geq
\mathop\mathrm{tr}\nolimits^{1/k}(\bm{B}^{k}),
\qquad
k=1,\ldots,n;
\\
\beta
\geq
\bigoplus_{m=1}^{n-k}
\alpha^{-m/k}
\mathop\mathrm{tr}\nolimits^{1/k}(\bm{F}_{km}),
\qquad
k=1,\ldots,n-1.
\end{gather*}

With properties of idempotent addition, we combine these inequalities into the system
\begin{equation}
\begin{gathered}
\alpha
\geq
\bigoplus_{k=1}^{n}
\mathop\mathrm{tr}\nolimits^{1/k}(\bm{A}^{k}),
\qquad
\beta
\geq
\bigoplus_{k=1}^{n}
\mathop\mathrm{tr}\nolimits^{1/k}(\bm{B}^{k}),
\\
\beta
\geq
\bigoplus_{k=1}^{n-1}
\bigoplus_{m=1}^{n-k}
\alpha^{-m/k}
\mathop\mathrm{tr}\nolimits^{1/k}(\bm{F}_{km}).
\end{gathered}
\label{I-alphageqtr1kAk}
\end{equation}

Furthermore, we examine the existence condition of Lemma~\ref{L-Axcleqxleqd}, which is represented in terms of the problem under consideration as follows:
\begin{equation*}
\bm{h}^{-}(\alpha^{-1}\bm{A}\oplus\beta^{-1}\bm{B})^{\ast}\bm{g}
\leq
\mathbb{1}.
\end{equation*}

Application of the expansion formula at \eqref{E-AoplusBast} yields
\begin{equation*}
\bigoplus_{k=0}^{n-1}
\alpha^{-k}
\bm{h}^{-}
\bm{A}^{k}
\bm{g}
\oplus
\bigoplus_{k=1}^{n-2}
\bigoplus_{m=1}^{n-k-1}
\alpha^{-m}\beta^{-k}
\bm{h}^{-}
\bm{F}_{km}
\bm{g}
\oplus
\bigoplus_{k=0}^{n-1}
\beta^{-k}
\bm{h}^{-}
\bm{B}^{k}
\bm{g}
\leq
\mathbb{1}.
\end{equation*}

Taking into account that $\bm{g}\leq\bm{h}$ by assumption and hence $\bm{h}^{-}\bm{g}\leq\mathbb{1}$, we replace the inequality by the system of inequalities
\begin{gather*}
\bigoplus_{k=1}^{n-1}
\alpha^{-k}
\bm{h}^{-}
\bm{A}^{k}
\bm{g}
\leq
\mathbb{1},
\qquad
\bigoplus_{k=1}^{n-1}
\beta^{-k}
\bm{h}^{-}
\bm{B}^{k}
\bm{g}
\leq
\mathbb{1},
\\
\bigoplus_{k=1}^{n-2}
\bigoplus_{m=1}^{n-k-1}
\alpha^{-m}\beta^{-k}
\bm{h}^{-}
\bm{F}_{km}
\bm{g}
\leq
\mathbb{1}.
\end{gather*}

Similarly as above, we solve the first inequality for $\alpha$ and the other two for $\beta$ to write
\begin{equation}
\begin{gathered}
\alpha
\geq
\bigoplus_{k=1}^{n-1}
(\bm{h}^{-}
\bm{A}^{k}
\bm{g})^{1/k},
\qquad
\beta
\geq
\bigoplus_{k=1}^{n-1}
(\bm{h}^{-}
\bm{B}^{k}
\bm{g})^{1/k},
\\
\beta
\geq
\bigoplus_{k=1}^{n-2}
\bigoplus_{m=1}^{n-k-1}
\alpha^{-m/k}
(\bm{h}^{-}
\bm{F}_{km}
\bm{g})^{1/k}.
\end{gathered}
\label{I-alphageqhAkg1k}
\end{equation}

By combining all lower bounds for $\alpha$ and $\beta$ at \eqref{I-alphageqtr1kAk} and \eqref{I-alphageqhAkg1k}, we obtain the inequalities
\begin{gather*}
\alpha
\geq
\bigoplus_{k=1}^{n}
\mathop\mathrm{tr}\nolimits^{1/k}(\bm{A}^{k})
\oplus
\bigoplus_{k=1}^{n-1}
(\bm{h}^{-}
\bm{A}^{k}
\bm{g})^{1/k}
=
\lambda\oplus\gamma,
\\
\beta
\geq
\bigoplus_{k=1}^{n}
\mathop\mathrm{tr}\nolimits^{1/k}(\bm{B}^{k})
\oplus
\bigoplus_{k=1}^{n-1}
(\bm{h}^{-}
\bm{B}^{k}
\bm{g})^{1/k}
=
\mu\oplus\delta,
\\
\beta
\geq
\bigoplus_{k=1}^{n-1}
\bigoplus_{m=1}^{n-k}
\alpha^{-m/k}
\mathop\mathrm{tr}\nolimits^{1/k}(\bm{F}_{km})
\oplus
\bigoplus_{k=1}^{n-2}
\bigoplus_{m=1}^{n-k-1}
\alpha^{-m/k}
(\bm{h}^{-}
\bm{F}_{km}
\bm{g})^{1/k}
=
G(\alpha).
\end{gather*}

We couple the last two inequalities to define the feasible domain for the variables $\alpha$ and $\beta$ by the system of inequalities
\begin{equation*}
\alpha
\geq
\lambda\oplus\gamma,
\qquad
\beta
\geq
\mu\oplus\delta
\oplus
G(\alpha).
\end{equation*}

Finally, as a result of the application of Lemma~\ref{L-Axcleqxleqd}, problem \eqref{P-minalphabeta-alpha1Aoplusbeta1Boplusgleqxleqh} decomposes into the bi-objective optimization problem
\begin{equation}
\begin{aligned}
\min_{\alpha,\beta}
&&&
(\alpha,\beta);
\\
\text{s.t.}
&&&
\alpha
\geq
\lambda\oplus\gamma,
\qquad
\beta
\geq
\mu\oplus\delta
\oplus
G(\alpha);
\end{aligned}
\label{P-minalphabeta-alphageqlambdaoplusgamma-betageqmuoplusdeltsoplusGalpha}
\end{equation}
together with a parametric description of solution vectors in the form
\begin{equation*}
\bm{x}
=
(\alpha^{-1}\bm{A}
\oplus
\beta^{-1}\bm{B})^{\ast}
\bm{u},
\qquad
\bm{g}
\leq
\bm{u}
\leq
(\bm{h}^{-}
(\alpha^{-1}\bm{A}
\oplus
\beta^{-1}\bm{B})^{\ast})^{-},
\end{equation*}
where $\alpha$ and $\beta$ are solutions of problem \eqref{P-minalphabeta-alphageqlambdaoplusgamma-betageqmuoplusdeltsoplusGalpha}.

It remains to solve the minimization problem at \eqref{P-minalphabeta-alphageqlambdaoplusgamma-betageqmuoplusdeltsoplusGalpha} in the Pareto-optimal sense. The solution involves the derivation of the Pareto front defined in the feasible domain as the pairs $(\alpha,\beta)$ in which no one variable can be decreased without increasing the other.

First, we observe that the function $G(\alpha)$ monotonically decreases as $\alpha$ increases. As a result, the feasible domain given by the constraints in \eqref{P-minalphabeta-alphageqlambdaoplusgamma-betageqmuoplusdeltsoplusGalpha} is the part of the $\alpha\beta$-plane, which is bounded from the left by the vertical line $\alpha=\lambda\oplus\gamma$, and from below by the horizontal line $\beta=\mu\oplus\delta$ and the curve $\beta=G(\alpha)$ (see Fig.~\ref{F-EPF} for an illustration).

To represent the Pareto front, we first consider the inequality $G(\alpha)<\mu\oplus\delta$, which determines those $\alpha$ where the graph of the function $G(\alpha)$ lies below the line $\beta=\mu\oplus\delta$. With the inverse function of $G$, denoted by $H$, the solution of the inequality for $\alpha$ takes the form $H(\mu\oplus\delta)<\alpha$.

Suppose that the condition $\lambda\oplus\gamma\geq H(\mu\oplus\delta)$ is satisfied. Then, for all $\alpha>\lambda\oplus\gamma$, we have $\alpha>H(\mu\oplus\delta)$, which means that for all feasible $\alpha$, the inequality $G(\alpha)<\mu\oplus\delta$ is valid. As a result, the Pareto front reduces to a single point with
\begin{equation*}
\alpha
=
\lambda\oplus\gamma,
\qquad
\beta
=
\mu\oplus\delta.
\end{equation*}

If the condition $\lambda\oplus\gamma<H(\mu\oplus\delta)$ holds, then $G(\alpha)\geq\mu\oplus\delta$ for all $\alpha$ such that $\lambda\oplus\gamma\leq\alpha\leq H(\mu\oplus\delta)$. In this case, the Pareto front takes the form of a segment, where
\begin{equation*}
\lambda\oplus\gamma
\leq
\alpha
\leq
H(\mu\oplus\delta),
\qquad
\beta
=
G(\alpha).
\end{equation*}

The derivation of the Pareto front completes the proof.
\qed
\end{proof}

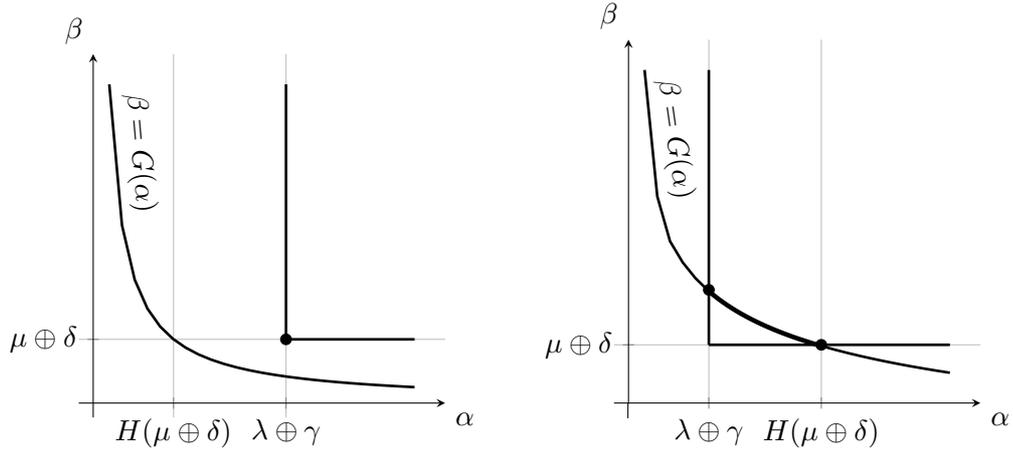
\begin{figure}[ht]
\pgfplotsset{
    standard/.style={
        axis x line=middle,
        axis y line=middle,
        enlarge x limits=0.1,
        enlarge y limits=0.1,
        every axis x label/.style={at={(current axis.right of origin)},anchor=north west},
        every axis y label/.style={at={(current axis.above origin)},anchor=south east}				
    }
}
\begin{tikzpicture}
\begin{axis}[
standard,
grid=major,
width=64mm,
height=64mm,
xlabel=$\alpha$,
ylabel=$\beta$,
xtick={0,5,12},
xticklabels={0,$H(\mu\oplus\delta)$,$\lambda\oplus\gamma$},
ytick={0,0.2},
yticklabels={0,$\mu\oplus\delta$},
]

\addplot[
black,
line width=1.0pt,
domain=1:20,
y domain=1:10,
]
{max(1/x^(3/2),1/x,1/x^(1/2)/50}
;

\addplot[
black,
line width=1.0pt,
]
coordinates {(12,0.2) (12,1.0)};

\addplot[
black,
line width=1.0pt,
]
coordinates {(12,0.2) (20,0.2)};

\addplot[only marks,
mark options={black},
]
coordinates {(12,0.2)};
\end{axis}

\node [rotate=-85] at (0.8,3.5) {$\beta=G(\alpha)$};


\end{tikzpicture}
\hspace{5mm}
\begin{tikzpicture}
\begin{axis}[
standard,
grid=major,
width=64mm,
height=64mm,
xlabel=$\alpha$,
ylabel=$\beta$,
xtick={0,5,12},
xticklabels={0,$\lambda\oplus\gamma$,$H(\mu\oplus\delta)$},
ytick={0,0.45},
yticklabels={0,$\mu\oplus\delta$},
]

\addplot[
black,
line width=1.0pt,
domain=1:20,
y domain=1:10,
]
{max(1/x^(1/2),1/x^(1/4)/1.2}
;

\addplot[
black,
line width=1.75pt,
domain=5:12,
y domain=1:10,
]
{max(1/x^(1/2),1/x^(1/4)/1.2};

\addplot[
black,
line width=1.0pt,
]
coordinates {(5,0.45) (5,1.0)};

\addplot[
black,
line width=1.0pt,
]
coordinates {(5,0.45) (20,0.45)};

\addplot[only marks,
mark options={black},
]
coordinates {(5,0.56) (12,0.45)};
\end{axis}

\draw (0.185,0) -- (0.185,-0.2);

\node [rotate=-82] at (0.8,3.5) {$\beta=G(\alpha)$};


\end{tikzpicture}
\caption{Examples of Pareto front given by a single point (left) and a segment (right).}
\label{F-EPF}
\end{figure}

\subsection{Solution of two-dimensional problems}

To illustrate application of Theorem~\ref{T-minxAxxBx-gleqxleqh}, we first derive a Pareto-optimal solution of a two-dimensional problem with arbitrary matrices, and then exploit this solution to handle a problem with symmetrically reciprocal matrices.
\begin{example}
\label{X-Aeqa11a12a21a22}
Consider problem \eqref{P-minxAxxBx-gleqxleqh} with $n=2$, where the matrices and vectors are given by
\begin{equation*}
\bm{A}
=
\begin{pmatrix}
a_{11} & a_{12}
\\
a_{21} & a_{22}
\end{pmatrix},
\qquad
\bm{B}
=
\begin{pmatrix}
b_{11} & b_{12}
\\
b_{21} & b_{22}
\end{pmatrix},
\qquad
\bm{g}
=
\begin{pmatrix}
g_{1}
\\
g_{2}
\end{pmatrix},
\qquad
\bm{h}
=
\begin{pmatrix}
h_{1}
\\
h_{2}
\end{pmatrix}.
\end{equation*}

To find all regular vectors $\bm{x}=(x_{1},x_{2})^{T}$ that solve the problem, we find the spectral radii of the matrices $\bm{A}$ and $\bm{B}$. Specifically, after calculating the matrix
\begin{equation*}
\bm{A}^{2}
=
\begin{pmatrix}
a_{11}^{2}\oplus a_{12}a_{21} & a_{12}(a_{11}\oplus a_{22})
\\
a_{21}(a_{11}\oplus a_{22}) & a_{12}a_{21}\oplus a_{22}^{2}
\end{pmatrix},
\end{equation*}
the spectral radius of $\bm{A}$ is found as
\begin{equation*}
\lambda
=
\mathop\mathrm{tr}\bm{A}
\oplus
\mathop\mathrm{tr}\nolimits^{1/2}(\bm{A}^{2})
=
a_{11}\oplus a_{22}\oplus a_{12}^{1/2}a_{21}^{1/2}.
\end{equation*}

In a similar way, we obtain
\begin{equation*}
\mu
=
\mathop\mathrm{tr}\bm{B}
\oplus
\mathop\mathrm{tr}\nolimits^{1/2}(\bm{B}^{2})
=
b_{11}\oplus b_{22}\oplus b_{12}^{1/2}b_{21}^{1/2}.
\end{equation*}

Next, we calculate the scalars
\begin{gather*}
\gamma
=
\bm{h}^{-}\bm{A}\bm{g}
=
h_{1}^{-1}a_{11}g_{1}
\oplus
h_{1}^{-1}a_{12}g_{2}
\oplus
h_{2}^{-1}a_{21}g_{1}
\oplus
h_{2}^{-1}a_{22}g_{2},
\\
\delta
=
\bm{h}^{-}\bm{B}\bm{g}
=
h_{1}^{-1}b_{11}g_{1}
\oplus
h_{1}^{-1}b_{12}g_{2}
\oplus
h_{2}^{-1}b_{21}g_{1}
\oplus
h_{2}^{-1}b_{22}g_{2}.
\end{gather*}

To derive the functions $G$ and $H$, we first evaluate
\begin{equation*}
\bm{F}_{11}
=
\bm{A}\bm{B}
\oplus
\bm{B}\bm{A},
\qquad
\mathop\mathrm{tr}\bm{F}_{11}
=
\mathop\mathrm{tr}(\bm{A}\bm{B}),
\end{equation*}
where
\begin{equation*}
\mathop\mathrm{tr}(\bm{A}\bm{B})
=
a_{11}b_{11}\oplus a_{12}b_{21}
\oplus
a_{21}b_{12}\oplus a_{22}b_{22}.
\end{equation*}

With $n=2$, the functions reduce to
\begin{equation*}
G(s)
=
s^{-1}\mathop\mathrm{tr}(\bm{A}\bm{B}),
\qquad
H(t)
=
t^{-1}\mathop\mathrm{tr}(\bm{A}\bm{B}).
\end{equation*}

Finally, we consider the Kleene star matrix, which generates the solutions of the problem and takes the form
\begin{equation*}
(\alpha^{-1}\bm{A}\oplus\beta^{-1}\bm{B})^{\ast}
=
\bm{I}
\oplus
\alpha^{-1}\bm{A}\oplus\beta^{-1}\bm{B}.
\end{equation*}

Since $\alpha\geq\lambda\oplus\gamma\geq a_{ii}$ and $\beta\geq\mu\oplus\delta\geq b_{ii}$, we have $\alpha^{-1}a_{ii}\leq\mathbb{1}$ and $\beta^{-1}b_{ii}\leq\mathbb{1}$ for $i=1,2$. As a result, the Kleene matrix becomes
\begin{equation*}
(\alpha^{-1}\bm{A}\oplus\beta^{-1}\bm{B})^{\ast}
=
\begin{pmatrix}
\mathbb{1} & \alpha^{-1}a_{12}\oplus\beta^{-1}b_{12}
\\
\alpha^{-1}a_{21}\oplus\beta^{-1}b_{21} & \mathbb{1}
\end{pmatrix}.
\end{equation*}

We now describe all Pareto-optimal solutions of the problem according to Theorem~\ref{T-minxAxxBx-gleqxleqh}. Suppose that the condition $\mathop\mathrm{tr}(\bm{A}\bm{B})<(\lambda\oplus\gamma)(\mu\oplus\delta)$ holds. In this case, the Pareto front of the problem contracts to the point $(\alpha,\beta)$ with
\begin{equation*}
\alpha
=
\lambda\oplus\gamma,
\qquad
\beta
=
\mu\oplus\delta.
\end{equation*}

If $\mathop\mathrm{tr}(\bm{A}\bm{B})\geq(\lambda\oplus\gamma)(\mu\oplus\delta)$, then the Pareto front is the segment given by
\begin{equation*}
\lambda\oplus\gamma
\leq
\alpha
\leq
(\mu\oplus\delta)^{-1}\mathop\mathrm{tr}(\bm{A}\bm{B}),
\qquad
\beta
=
\alpha^{-1}\mathop\mathrm{tr}(\bm{A}\bm{B}).
\end{equation*}

The Pareto-optimal solutions are represented in parametric form as
\begin{equation*}
\bm{x}
=
\begin{pmatrix}
\mathbb{1} & \alpha^{-1}a_{12}\oplus\beta^{-1}b_{12}
\\
\alpha^{-1}a_{21}\oplus\beta^{-1}b_{21} & \mathbb{1}
\end{pmatrix}\bm{u},
\end{equation*}
where the vector of parameters $\bm{u}=(u_{1},u_{2})^{T}$ satisfies the conditions
\begin{equation*}
\begin{pmatrix}
g_{1}
\\
g_{2}
\end{pmatrix}
\leq
\bm{u}
\leq
\begin{pmatrix}
(h_{1}^{-1}\oplus h_{2}^{-1}(\alpha^{-1}a_{21}\oplus\beta^{-1}b_{21}))^{-1}
\\
(h_{1}^{-1}(\alpha^{-1}a_{12}\oplus\beta^{-1}b_{12})\oplus h_{2}^{-1})^{-1}
\end{pmatrix}.
\end{equation*}
\end{example}

\begin{example}
Consider the previous example in the framework of the max-algebra $\mathbb{R}_{\max}$, where the matrices $\bm{A}$ and $\bm{B}$ are symmetrically reciprocal. Let us define
\begin{equation*}
\bm{A}
=
\begin{pmatrix}
1 & 2
\\
1/2 & 1
\end{pmatrix},
\qquad
\bm{B}
=
\begin{pmatrix}
1 & 1/3
\\
3 & 1
\end{pmatrix},
\qquad
\bm{g}
=
\begin{pmatrix}
1/3
\\
1/3
\end{pmatrix},
\qquad
\bm{h}
=
\begin{pmatrix}
1/2
\\
1/2
\end{pmatrix}.
\end{equation*}

First, we observe that the assumptions of Theorem~\ref{T-minxAxxBx-gleqxleqh} are fulfilled. We calculate the matrices and vectors
\begin{equation*}
\bm{A}^{2}
=
\bm{A},
\qquad
\bm{B}^{2}
=
\bm{B},
\qquad
\bm{A}\bm{B}
=
\begin{pmatrix}
6 & 2
\\
3 & 1
\end{pmatrix},
\qquad
\bm{A}\bm{g}
=
\begin{pmatrix}
2/3
\\
1/3
\end{pmatrix},
\qquad
\bm{B}\bm{g}
=
\begin{pmatrix}
1
\\
3/2
\end{pmatrix},
\end{equation*}
and then evaluate the traces
\begin{equation*}
\mathop\mathrm{tr}\bm{A}
=
\mathop\mathrm{tr}\bm{A}^{2}
=
1,
\qquad
\mathop\mathrm{tr}\bm{B}
=
\mathop\mathrm{tr}\bm{B}^{2}
=
1,
\qquad
\mathop\mathrm{tr}(\bm{A}\bm{B})
=
6.
\end{equation*}

With the above results, we obtain
\begin{equation*}
\lambda
=
1,
\qquad
\mu
=
1,
\qquad
\gamma
=
4/3,
\qquad
\delta
=
2,
\qquad
\lambda\oplus\gamma
=
4/3,
\qquad
\mu\oplus\delta
=
2.
\end{equation*}

Since the condition $\mathop\mathrm{tr}(\bm{A}\bm{B})=6>(\lambda\oplus\gamma)(\mu\oplus\delta)=8/3$ holds, the Pareto front takes the form of the segment
\begin{equation*}
4/3
\leq
\alpha
\leq
3,
\qquad
\beta
=
6\alpha^{-1},
\end{equation*}
whereas all Pareto-optimal solutions are given by
\begin{equation*}
\bm{x}
=
\begin{pmatrix}
1 & 2\alpha^{-1}\oplus3^{-1}\beta^{-1}
\\
2^{-1}\alpha^{-1}\oplus3\beta^{-1} & 1
\end{pmatrix}\bm{u},
\qquad
\bm{u}>\bm{0}.
\end{equation*}

Let us verify that $2\alpha^{-1}\oplus3^{-1}\beta^{-1}=2\alpha^{-1}$. Indeed, after substitution of $\beta=6\alpha^{-1}$, we can write $3^{-1}\beta^{-1}=18^{-1}\alpha\leq6^{-1}<2/3\leq2\alpha^{-1}$. In a similar way, we find that $2^{-1}\alpha^{-1}\oplus3\beta^{-1}=3\beta^{-1}=2^{-1}\alpha$. As a result, the Kleene star matrix reduces to
\begin{equation*}
\begin{pmatrix}
1 & 2\alpha^{-1}
\\
2^{-1}\alpha & 1
\end{pmatrix}.
\end{equation*}

Since both columns in this matrix are collinear, one of them can be eliminated to represent all solution as
\begin{equation*}
\bm{x}
=
\begin{pmatrix}
1
\\
2^{-1}\alpha
\end{pmatrix}u,
\qquad
u>0,
\qquad
4/3
\leq
\alpha
\leq
3.
\end{equation*}

An illustration is given in Fig.~\ref{F-PFPOS}, where the Pareto front appears as a thick segment (left), whereas the Pareto-optimal solutions $\bm{x}$ form a cone spanned by the vectors $(1,2/3)^{T}$ and $(1,3/2)^{T}$, which correspond to the limiting points of the front (right).
\begin{figure}[ht]
\pgfplotsset{
    standard/.style={
        axis x line=middle,
        axis y line=middle,
        enlarge x limits=0.1,
        enlarge y limits=0.1,
        every axis x label/.style={at={(current axis.right of origin)},anchor=north west},
        every axis y label/.style={at={(current axis.above origin)},anchor=south east}				
    }
}
\begin{tikzpicture}
\begin{axis}[
standard,
grid=major,
width=64mm,
height=64mm,
xlabel=$\alpha$,
ylabel=$\beta$,
xtick={0,1.333,3},
xticklabels={0,$4/3$,$3$},
ytick={0,2},
yticklabels={0,$2$},
]

\addplot[
black,
line width=1.0pt,
domain=1:4,
y domain=1:3,
]
{6/x}
;

\addplot[
black,
line width=1.75pt,
domain=1.333:3,
y domain=1:3,
]
{6/x};



\addplot[only marks,
mark options={black},
]
coordinates {(1.333,4.5) (3,2)};
\end{axis}

\draw (0,0) -- (0,-0.2);
\draw (0,0) -- (-0.2,0);

\node [rotate=-45] at (2.1,2.0) {$\beta=6\alpha^{-1}$};

\end{tikzpicture}
\hspace{5mm}
\begin{tikzpicture}

\draw [->] (0,0.2) -- (5.0,0.2);
\draw [->] (0.2,0) -- (0.2,5.0);

\node at (-0.1,5.3) {$x_{2}$};
\node at (5.3,0.0) {$x_{1}$};

\draw (2,0.1) -- (2,3);
\node at (2,-0.15) {$1$};

\draw (0.1,1.333) -- (2,1.333);
\node at (-0.25,1.333) {$2/3$};

\draw (0.1,3) -- (2,3);
\node at (-0.25,3) {$3/2$};

\draw [thick] (0.2,0.2) -- (2.5,3.75);
\draw [ultra thick,->] (0.2,0.2) -- (2.0,3.0);

\draw [thick] (0.2,0.2) -- (4.0,2.6);
\draw [ultra thick,->] (0.2,0.2) -- (2.0,1.35);

\draw [ultra thick,->] (0.2,0.2) -- (3.0,3.0);

\node at (3.3,3.1) {$\bm{x}$};
\end{tikzpicture}
\caption{Pareto front (left) and Pareto-optimal solution (right).}
\label{F-PFPOS}
\end{figure}
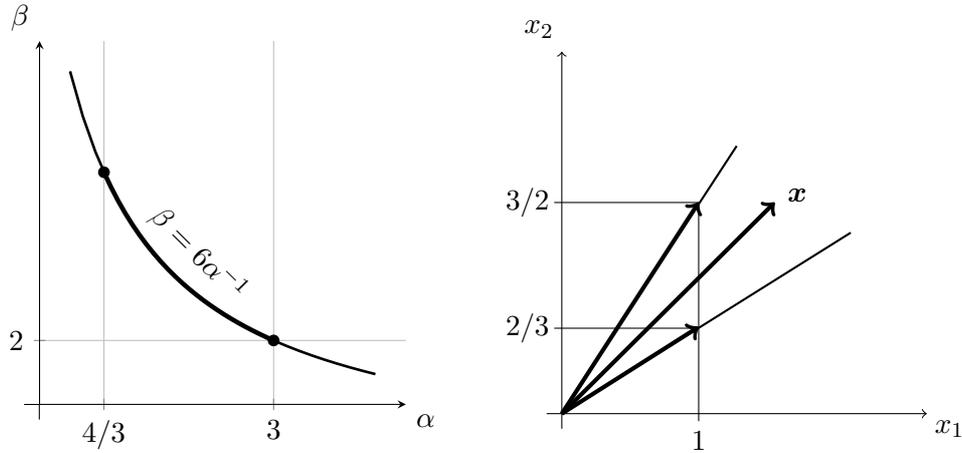
\end{example}

\section{Application to bi-criteria decision problem}
\label{S-ABCBCDP}

We are now in a position to solve the box-constrained bi-criteria decision problem at \eqref{P-min_maxijaijxixj-maxijbijxixj-gileqxilehi}. In the framework of the max-algebra $\mathbb{R}_{\max}$, the problem becomes
\begin{equation*}
\begin{aligned}
\min_{x_{1},\ldots,x_{n}}
&&&
\left(
\bigoplus_{1\leq i,j\leq n}x_{i}^{-1}a_{ij}x_{j},
\bigoplus_{1\leq i,j\leq n}x_{i}^{-1}b_{ij}x_{j}
\right);
\\
\text{s.t.}
&&&
g_{j}
\leq
x_{j}
\leq
h_{j},
\qquad
j=1,\ldots,n.
\end{aligned}
\end{equation*}

Furthermore, we introduce the matrix-vector notation 
\begin{equation*}
\bm{A}
=
(a_{ij}),
\qquad
\bm{B}
=
(b_{ij}),
\qquad
\bm{g}
=
(g_{j}),
\qquad
\bm{h}
=
(h_{j}),
\qquad
\bm{x}
=
(x_{j})
\end{equation*}
to represent the vector objective function as
\begin{equation*}
(\bm{x}^{-}\bm{A}\bm{x},\ \bm{x}^{-}\bm{B}\bm{x}),
\end{equation*}
and the inequality constraints as
\begin{equation*}
\bm{g}
\leq
\bm{x}
\leq
\bm{h}.
\end{equation*}

We combine the objective function and inequality constraint, which yields a constrained bi-objective optimization problem in the form of \eqref{P-minxAxxBx-gleqxleqh}. We observe that in the context of decision making, the matrices $\bm{A}$ and $\bm{B}$ are indeed nonzero and the vector $\bm{h}$ is regular. In this case, the solution of the problem is given by Theorem~\ref{T-minxAxxBx-gleqxleqh}, which is applied in the max-algebra setting.

As an illustration of the computational technique used in the solution, we present a numerical example of the solution of a problem with four alternatives. To see how the constraints may affect the solution, we take as the starting point the unconstrained problem examined in \cite{Krivulin2020Using}, and then add box constraints on the solution vector to form a new constrained problem.

\begin{example}
Suppose one needs to evaluate the absolute ratings of $n=4$ alternatives that are compared in pairs with respect to two criteria, subject to box constraints imposed on the ratings. The matrices of pairwise comparisons and bound vectors of constraints are given by
\begin{gather*}
\bm{A}
=
\begin{pmatrix}
1 & 3 & 4 & 2
\\
1/3 & 1 & 1/2 & 1/3
\\
1/4 & 2 & 1 & 4
\\
1/2 & 3 & 1/4 & 1
\end{pmatrix},
\qquad
\bm{B}
=
\begin{pmatrix}
1 & 2 & 4 & 2
\\
1/2 & 1 & 1/3 & 1/2
\\
1/4 & 3 & 1 & 4
\\
1/2 & 2 & 1/4 & 1
\end{pmatrix},
\\
\bm{g}
=
\begin{pmatrix}
1
\\
0
\\
0
\\
0
\end{pmatrix},
\quad
\bm{h}
=
\begin{pmatrix}
1
\\
1/6
\\
1
\\
1
\end{pmatrix}.
\end{gather*}

A complete solution to the problem when the box constraint $\bm{g}\leq\bm{x}\leq\bm{h}$ is not taken into account is given in \cite{Krivulin2020Using}, where the Pareto front in the $\alpha\beta$ plane forms the segment
\begin{equation*}
2
\leq
\alpha
\leq
3,
\qquad
\beta
=
24\alpha^{-3}\oplus24^{1/3}\alpha^{-1/3},
\end{equation*}
and the set of Pareto-optimal solutions $\bm{x}$ is described using a vector of parameters $\bm{u}$ as
\begin{equation*}
\bm{x}
=
(\alpha^{-1}\bm{A}\oplus\beta^{-1}\bm{B})^{\ast}
\bm{u},
\qquad
\bm{u}
>
\bm{0}.
\end{equation*}

Specifically, at the limiting points of the front with $\alpha_{1}=2$ and $\alpha_{2}=3$, the solution reduces to the vectors
\begin{equation*}
\bm{x}_{1}
=
\begin{pmatrix}
1
\\
1/6
\\
1/2
\\
1/4
\end{pmatrix}
u,
\qquad
u>0;
\qquad
\bm{x}_{2}
=
\begin{pmatrix}
1
\\
1/4
\\
1/2
\\
1/4
\end{pmatrix}
v,
\qquad
v>0.
\end{equation*}

We now turn to the box-constrained problem in complete form \eqref{P-minxAxxBx-gleqxleqh}. We exploit the results obtained in \cite{Krivulin2020Using} to write the matrices and spectral radius
\begin{gather*}
\bm{A}^{2}
=
\begin{pmatrix}
1 & 8 & 4 & 16
\\
1/3 & 1 & 4/3 & 2
\\
2 & 12 & 1 & 4
\\
1 & 3 & 2 & 1
\end{pmatrix},
\qquad
\bm{A}^{3}
=
\begin{pmatrix}
8 & 48 & 4 & 16
\\
1 & 6 & 4/3 & 16/3
\\
4 & 12 & 8 & 4
\\
1 & 4 & 4 & 8
\end{pmatrix},
\\
\bm{A}^{4}
=
\begin{pmatrix}
16 & 48 & 32 & 16
\\
8/3 & 16 & 4 & 16/3
\\
4 & 16 & 16 & 32
\\
4 & 24 & 4 & 16
\end{pmatrix},
\qquad
\lambda
=
2.
\end{gather*}

Similarly, we have the matrices and spectral radius
\begin{gather*}
\bm{B}^{2}
=
\begin{pmatrix}
1 & 12 & 4 & 16
\\
1/2 & 1 & 2 & 4/3
\\
2 & 8 & 1 & 4
\\
1 & 2 & 2 & 1
\end{pmatrix},
\qquad
\bm{B}^{3}
=
\begin{pmatrix}
8 & 32 & 4 & 16
\\
2/3 & 6 & 2 & 8
\\
4 & 8 & 8 & 4
\\
1 & 6 & 4 & 8
\end{pmatrix},
\\
\bm{B}^{4}
=
\begin{pmatrix}
16 & 32 & 32 & 16
\\
4 & 16 & 8/3 & 8
\\
4 & 24 & 16 & 32
\\
4 & 16 & 4 & 16
\end{pmatrix},
\qquad
\mu
=
2.
\end{gather*}

Next, we calculate the vectors
\begin{equation*}
\bm{A}\bm{g}
=
\begin{pmatrix}
1
\\
1/3
\\
1/4
\\
1/2
\end{pmatrix},
\qquad
\bm{A}^{2}\bm{g}
=
\begin{pmatrix}
1
\\
1/3
\\
2
\\
1
\end{pmatrix},
\qquad
\bm{A}^{3}\bm{g}
=
\begin{pmatrix}
8
\\
1
\\
4
\\
1
\end{pmatrix},
\end{equation*}
and then find
\begin{equation*}
\bm{h}^{-}\bm{A}\bm{g}
=
2,
\qquad
(\bm{h}^{-}\bm{A}^{2}\bm{g})^{1/2}
=
2^{1/2},
\qquad
(\bm{h}^{-}\bm{A}^{3}\bm{g})^{1/3}
=
2.
\end{equation*}

Combining these results yields
\begin{equation*}
\gamma
=
\bm{h}^{-}\bm{A}\bm{g}
\oplus
(\bm{h}^{-}\bm{A}^{2}\bm{g})^{1/2}
\oplus
(\bm{h}^{-}\bm{A}^{3}\bm{g})^{1/3}
=
2.
\end{equation*}

A similar computation gives
\begin{equation*}
\delta
=
\bm{h}^{-}\bm{B}\bm{g}
\oplus
(\bm{h}^{-}\bm{B}^{2}\bm{g})^{1/2}
\oplus
(\bm{h}^{-}\bm{B}^{3}\bm{g})^{1/3}
=
3.
\end{equation*}

We now examine which condition $H(\mu\oplus\delta)\leq\lambda\oplus\gamma$ and $H(\mu\oplus\delta)>\lambda\oplus\gamma$ in Theorem~\ref{T-minxAxxBx-gleqxleqh} is satisfied. First, we write the function
\begin{multline*}
H(t)
=
t^{-1}
\mathop\mathrm{tr}(\bm{F}_{11})
\oplus
t^{-1/2}
\mathop\mathrm{tr}\nolimits^{1/2}(\bm{F}_{12})
\oplus
t^{-1/3}
\mathop\mathrm{tr}\nolimits^{1/3}(\bm{F}_{13})
\\
\oplus
t^{-2}
\mathop\mathrm{tr}(\bm{F}_{21})
\oplus
t^{-1}
\mathop\mathrm{tr}\nolimits^{1/2}(\bm{F}_{22})
\oplus
t^{-3}
\mathop\mathrm{tr}(\bm{F}_{31})
\\
\oplus
t^{-1}
\bm{h}^{-}
\bm{F}_{11}
\bm{g}
\oplus
t^{-1/2}
(\bm{h}^{-}
\bm{F}_{12}
\bm{g})^{1/2}
\oplus
t^{-2}
\bm{h}^{-}
\bm{F}_{21}
\bm{g}.
\end{multline*}

Furthermore, we define the matrices
\begin{gather*}
\bm{F}_{11}
=
\bm{A}\bm{B}
\oplus
\bm{B}\bm{A},
\qquad
\bm{F}_{12}
=
\bm{A}^{2}\bm{B}
\oplus
\bm{A}\bm{B}\bm{A}
\oplus
\bm{B}\bm{A}^{2},
\\
\bm{F}_{13}
=
\bm{A}^{3}\bm{B}
\oplus
\bm{A}^{2}\bm{B}\bm{A}
\oplus
\bm{A}\bm{B}\bm{A}^{2}
\oplus
\bm{B}\bm{A}^{3},
\qquad
\bm{F}_{21}
=
\bm{A}\bm{B}^{2}
\oplus
\bm{B}\bm{A}\bm{B}
\oplus
\bm{B}^{2}\bm{A},
\\
\bm{F}_{22}
=
\bm{A}^{2}\bm{B}^{2}
\oplus
(\bm{A}\bm{B})^{2}
\oplus
\bm{B}\bm{A}^{2}\bm{B}
\oplus
(\bm{B}\bm{A})^{2}
\oplus
\bm{B}^{2}\bm{A}^{2},
\\
\bm{F}_{31}
=
\bm{A}\bm{B}^{3}
\oplus
\bm{B}\bm{A}\bm{B}^{2}
\oplus
\bm{B}^{2}\bm{A}\bm{B}
\oplus
\bm{B}^{3}\bm{A}.
\end{gather*}

Then, we need to calculate the matrix products
\begin{gather*}
\bm{A}\bm{B}
=
\begin{pmatrix}
3/2 & 12 & 4 & 16
\\
1/2 & 3/2 & 4/3 & 2
\\
2 & 8 & 1 & 4
\\
3/2 & 3 & 2 & 3/2
\end{pmatrix},
\qquad
\bm{B}\bm{A}
=
\begin{pmatrix}
1 & 8 & 4 & 16
\\
1/2 & 3/2 & 2 & 4/3
\\
2 & 12 & 3/2 & 4
\\
2/3 & 3 & 2 & 1
\end{pmatrix},
\\
\bm{A}^{2}\bm{B}
=
\begin{pmatrix}
8 & 32 & 4 & 16
\\
1 & 4 & 4/3 & 16/3
\\
6 & 12 & 8 & 6
\\
3/2 & 6 & 4 & 8
\end{pmatrix},
\qquad
\bm{A}\bm{B}\bm{A}
=
\begin{pmatrix}
8 & 48 & 6 & 16
\\
1 & 6 & 2 & 16/3
\\
8/3 & 12 & 8 & 4
\\
3/2 & 9/2 & 6 & 8
\end{pmatrix},
\\
\bm{B}\bm{A}^{2}
=
\begin{pmatrix}
8 & 48 & 4 & 16
\\
2/3 & 4 & 2 & 8
\\
4 & 12 & 8 & 6
\\
1 & 4 & 8/3 & 8
\end{pmatrix},
\qquad
\bm{A}^{3}\bm{B}
=
\begin{pmatrix}
24 & 48 & 32 & 24
\\
3 & 32/3 & 4 & 16/3
\\
6 & 24 & 16 & 32
\\
4 & 16 & 4 & 16
\end{pmatrix},
\\
\bm{A}\bm{B}^{2}
=
\begin{pmatrix}
8 & 32 & 6 & 16
\\
1 & 4 & 2 & 16/3
\\
4 & 8 & 8 & 4
\\
3/2 & 6 & 6 & 8
\end{pmatrix},
\qquad
\bm{B}\bm{A}\bm{B}
=
\begin{pmatrix}
8 & 32 & 4 & 16
\\
3/4 & 6 & 2 & 8
\\
6 & 12 & 8 & 6
\\
3/2 & 6 & 8/3 & 8
\end{pmatrix},
\\
\bm{B}^{2}\bm{A}
=
\begin{pmatrix}
8 & 48 & 6 & 16
\\
2/3 & 4 & 2 & 8
\\
8/3 & 12 & 8 & 4
\\
1 & 4 & 4 & 8
\end{pmatrix},
\qquad
\bm{A}^{2}\bm{B}^{2}
=
\begin{pmatrix}
16 & 32 & 32 & 16
\\
8/3 & 32/3 & 4 & 16/3
\\
6 & 24 & 24 & 32
\\
4 & 16 & 6 & 16
\end{pmatrix},
\\
(\bm{A}\bm{B})^{2}
=
\begin{pmatrix}
24 & 48 & 32 & 24
\\
3 & 32/3 & 4 & 8
\\
6 & 24 & 32/3 & 32
\\
4 & 18 & 6 & 24
\end{pmatrix},
\qquad
\bm{A}\bm{B}^{3}
=
\begin{pmatrix}
16 & 32 & 32 & 24
\\
8/3 & 32/3 & 4 & 8
\\
4 & 24 & 16 & 32
\\
4 & 18 & 6 & 24
\end{pmatrix}.
\end{gather*}

We apply the additive and cyclic properties of trace to find
\begin{gather*}
\mathop\mathrm{tr}\bm{F}_{11}
=
\mathop\mathrm{tr}(\bm{A}\bm{B})
=
3/2,
\qquad
\mathop\mathrm{tr}\bm{F}_{12}
=
\mathop\mathrm{tr}(\bm{A}^{2}\bm{B})
=
8,
\qquad
\mathop\mathrm{tr}\bm{F}_{13}
=
\mathop\mathrm{tr}(\bm{A}^{3}\bm{B})
=
24,
\\
\mathop\mathrm{tr}\bm{F}_{21}
=
\mathop\mathrm{tr}(\bm{A}\bm{B}^{2})
=
8,
\qquad
\mathop\mathrm{tr}\bm{F}_{22}
=
\mathop\mathrm{tr}(\bm{A}^{2}\bm{B}^{2})
\oplus
\mathop\mathrm{tr}(\bm{A}\bm{B})^{2}
=
24,
\\
\mathop\mathrm{tr}\bm{F}_{31}
=
\mathop\mathrm{tr}(\bm{A}\bm{B}^{3})
=
24.
\end{gather*}

Finally, we calculate the matrices
\begin{gather*}
\bm{F}_{11}
=
\begin{pmatrix}
3/2 & 12 & 4 & 16
\\
1/2 & 3/2 & 2 & 2
\\
2 & 12 & 3/2 & 4
\\
3/2 & 3 & 2 & 3/2
\end{pmatrix},
\qquad
\bm{F}_{12}
=
\begin{pmatrix}
8 & 48 & 6 & 16
\\
1 & 6 & 2 & 8
\\
6 & 12 & 8 & 6
\\
3/2 & 6 & 6 & 8
\end{pmatrix},
\\
\bm{F}_{21}
=
\begin{pmatrix}
8 & 48 & 6 & 16
\\
1 & 6 & 2 & 8
\\
6 & 12 & 8 & 6
\\
3/2 & 6 & 6 & 8
\end{pmatrix},
\end{gather*}
and then evaluate the scalars
\begin{equation*}
\bm{h}^{-}
\bm{F}_{11}
\bm{g}
=
3,
\qquad
\bm{h}^{-}
\bm{F}_{12}
\bm{g}
=
8,
\qquad
\bm{h}^{-}
\bm{F}_{21}
\bm{g}
=
8.
\end{equation*}

After substitution of the obtained results followed by simplification, the function $H$ takes the form
\begin{equation*}
H(t)
=
24t^{-3}
\oplus
8t^{-2}
\oplus
24^{1/2}t^{-1}
\oplus
8^{1/2}t^{-1/2}
\oplus
24^{1/3}t^{-1/3}.
\end{equation*}

It remains to calculate $\lambda\oplus\gamma=2$, $\mu\oplus\delta=3$ and $H(\mu\oplus\delta)=2$, and then conclude that the equality $H(\mu\oplus\delta)=\lambda\oplus\gamma$ is valid. It follows from Theorem~\ref{T-minxAxxBx-gleqxleqh} that the Pareto front in this case is a single point
\begin{equation*}
\alpha
=
\lambda\oplus\gamma
=
2,
\qquad
\beta
=
\mu\oplus\delta
=
3.
\end{equation*}

With $\alpha=2$ and $\beta=3$, all Pareto-optimal solutions are given by
\begin{equation*}
\bm{x}
=
(2^{-1}\bm{A}\oplus3^{-1}\bm{B})^{\ast}
\bm{u},
\qquad
\bm{g}
\leq
\bm{u}
\leq
(\bm{h}^{-}
(2^{-1}\bm{A}
\oplus
3^{-1}\bm{B})^{\ast})^{-}.
\end{equation*}

To represent the solution in more detail, we calculate the matrix
\begin{equation*}
2^{-1}\bm{A}\oplus3^{-1}\bm{B}
=
\begin{pmatrix}
 1/2 & 3/2 & 2 & 1
\\
 1/6 & 1/2 & 1/4 & 1/6
\\
 1/8 & 1 & 1/2 & 2
\\
 1/4 & 3/2 & 1/8 & 1/2
\end{pmatrix}
\end{equation*}
together with its second and third powers
\begin{equation*}
\begin{pmatrix}
 1/4 & 2 & 1 & 4
\\
 1/12 & 1/4 & 1/3 & 1/2
\\
 1/2 & 3 & 1/4 & 1
\\
 1/4 & 3/4 & 1/2 & 1/4
\end{pmatrix},
\qquad
\begin{pmatrix}
 1 & 6 & 1/2 & 2
\\
 1/8 & 3/4 & 1/6 & 2/3
\\
 1/2 & 3/2 & 1 & 1/2
\\
 1/8 & 1/2 & 1/2 & 1
\end{pmatrix}.
\end{equation*}

We combine these matrices and the identity matrix to from the Kleene star matrix
\begin{equation*}
(2^{-1}\bm{A}\oplus3^{-1}\bm{B})^{\ast}
=
\begin{pmatrix}
 1 & 6 & 2 & 4
\\
 1/6 & 1 & 1/3 & 2/3
\\
 1/2 & 3 & 1 & 2
\\
 1/4 & 3/2 & 1/2 & 1
\end{pmatrix}.
\end{equation*}

The vector of parameters $\bm{u}$ must satisfies the condition
\begin{equation*}
\bm{g}
=
\begin{pmatrix}
1
\\
0
\\
0
\\
0
\end{pmatrix}
\leq
\bm{u}
\leq
(\bm{h}^{-}
(2^{-1}\bm{A}\oplus3^{-1}\bm{B})^{\ast})^{-}
=
\begin{pmatrix}
 1
\\
1/6
\\
1/2
\\
1/4
\end{pmatrix}.
\end{equation*}

It is easy to verify that both the lower and upper bounds on the vector $\bm{u}$ yield the same vector
\begin{equation*}
\bm{x}
=
\begin{pmatrix}
 1
\\
1/6
\\
1/2
\\
1/4
\end{pmatrix},
\end{equation*}
which therefore appears to be a single solution of the problem.

Note that this result corresponds to the limiting solution $\bm{x}_{1}$ of the problem in \cite{Krivulin2020Using}.
\end{example}

\section{Conclusions}
\label{S-C}

In this paper, we considered the following problem: given two matrices of pairwise comparisons of alternatives evaluated according to two criteria, find a vector of ratings (priorities, weights) of the alternatives subject to box constraints imposed on the ratings. Such a constrained bi-criteria pairwise comparison problem can occur in complex decision tasks, where the ratings are calculated to make decision about which alternative is more preferable to be chosen under constraints that can reflect prior information available to the decision maker about the absolute ratings of alternatives.

As the two most common approaches, the methods of analytical hierarchy process and weighted geometric mean are used to solve multicriteria pairwise comparison problems. Both methods offer a single solution vector; the first method applies a computational procedure to find the vector of ratings numerically, whereas the second method derives the vector analytically by solving a problem of matrix approximation in the log-Euclidean sense. These methods, which are quite computationally efficient for unconstrained problems, can hardly incorporate constraints on the ratings without sufficiently altering the solution mechanism and significantly complicating the calculation procedure.

To handle the constrained bi-criteria pairwise comparison problem, we have applied a solution technique that combines log-Chebyshev approximation with tropical optimization. First, we have represented this problem as a matrix approximation problem that is to find a common consistent matrix that simultaneously approximates in the log-Chebyshev sense both pairwise comparison matrices subject to the given constraints. The approximating matrix, which has the form of a symmetrically reciprocal matrix of unit rank, is determined by a positive vector that is taken as a solution vector of ratings.

Furthermore, we have formulated the constrained approximation problem in terms of tropical mathematics (which studies algebraic systems with idempotent operations) as a tropical bi-objective optimization problem. We have solved the problem by using methods and results of tropical optimization, which yields all Pareto-optimal solutions given in an explicit analytical form suitable for further analysis and numerical computations.

The results obtained show that the approach based on log-Chebyshev approximation combined with tropical optimization offers a good potential to provide analytical solutions to both unconstrained and constrained multicriteria pairwise comparison problems. As an efficient instrument for solving pairwise comparison problems, this approach can be used to supplement and complement existing methods of evaluating alternatives in decision making, in particular for handling constrained problems.

At the same time, the proposed solution deals with only two criteria for comparisons, which substantially provides that all Pareto-optimal solutions be obtained analytically in a compact vector form. As one can see, the extension of the solution to problems with more criteria will lead to a considerable increase in the complexity of calculations, which makes it necessary to modify the technique. Specifically, since finding all Pareto-optimal solutions in pairwise comparison problems with many criteria may be too difficult, it seems to be more reasonable to concentrate on deriving a few particular solutions.

We consider the further development and improvement of the solution technique to handle constrained pairwise comparison problems with three and more criteria as a promising direction of future research. The evaluation of computational complexity of the proposed solutions presents another research problem of interest.

\bibliographystyle{abbrvurl}

\bibliography{Algebraic_solution_to_box-constrained_bi-criteria_problem_of_rating_alternatives_through_pairwise_comparisons}

\end{document}